\documentclass[10pt]{article}
\usepackage{amsmath}
\usepackage{amssymb}
\usepackage{amsthm}
\usepackage{mathrsfs}
\numberwithin{equation}{section}

\begin{document}
\title{Boundedness of intrinsic square functions on generalized Morrey spaces}
\author{Hua Wang \footnote{E-mail address: wanghua@pku.edu.cn.}\\
\footnotesize{Department of Mathematics, Zhejiang University, Hangzhou 310027, China}}
\date{}
\maketitle

\begin{abstract}
In this paper, we will study the strong type and weak type estimates of intrinsic square functions including the Lusin area integral, Littlewood-Paley $g$-function and $g^*_\lambda$-function on the generalized Morrey spaces $L^{p,\Phi}$ for $1\le p<\infty$, where $\Phi$ is a growth function on $(0,\infty)$ satisfying the doubling condition. The boundedness of the commutators generated by $BMO(\mathbb R^n)$ functions and intrinsic square functions is also obtained.
\end{abstract}

\section{Introduction and main results}

Let ${\mathbb R}^{n+1}_+=\mathbb R^n\times(0,\infty)$ and $\varphi_t(x)=t^{-n}\varphi(x/t)$. The classical square function (Lusin area integral) is a familiar object. If $u(x,t)=P_t*f(x)$ is the Poisson integral of $f$, where $P_t(x)=c_n\frac{t}{(t^2+|x|^2)^{{(n+1)}/2}}$ denotes the Poisson kernel in ${\mathbb R}^{n+1}_+$. Then we define the classical square function (Lusin area integral) $S(f)$ by (see \cite{stein})
\begin{equation*}
S(f)(x)=\bigg(\iint_{\Gamma(x)}\big|\nabla u(y,t)\big|^2t^{1-n}\,dydt\bigg)^{1/2},
\end{equation*}
where $\Gamma(x)$ denotes the usual cone of aperture one:
\begin{equation*}
\Gamma(x)=\big\{(y,t)\in{\mathbb R}^{n+1}_+:|x-y|<t\big\}
\end{equation*}
and
\begin{equation*}
\big|\nabla u(y,t)\big|=\left|\frac{\partial u}{\partial t}\right|^2+\sum_{j=1}^n\left|\frac{\partial u}{\partial y_j}\right|^2.
\end{equation*}
Similarly, we can define a cone of aperture $\beta$ for any $\beta>0$:
\begin{equation*}
\Gamma_\beta(x)=\big\{(y,t)\in{\mathbb R}^{n+1}_+:|x-y|<\beta t\big\},
\end{equation*}
and corresponding square function
\begin{equation*}
S_\beta(f)(x)=\bigg(\iint_{\Gamma_\beta(x)}\big|\nabla u(y,t)\big|^2t^{1-n}\,dydt\bigg)^{1/2}.
\end{equation*}
The Littlewood-Paley $g$-function (could be viewed as a ``zero-aperture" version of $S(f)$) and the $g^*_\lambda$-function (could be viewed as an ``infinite aperture" version of $S(f)$) are defined respectively by
\begin{equation*}
g(f)(x)=\bigg(\int_0^\infty\big|\nabla u(x,t)\big|^2 t\,dt\bigg)^{1/2}
\end{equation*}
and
\begin{equation*}
g^*_\lambda(f)(x)=\left(\iint_{{\mathbb R}^{n+1}_+}\bigg(\frac t{t+|x-y|}\bigg)^{\lambda n}\big|\nabla u(y,t)\big|^2 t^{1-n}\,dydt\right)^{1/2}, \quad \lambda>1.
\end{equation*}

The modern (real-variable) variant of $S_\beta(f)$ can be defined in the following way (here we drop the subscript $\beta$ if $\beta=1$). Let $\psi\in C^\infty(\mathbb R^n)$ be real, radial, have support contained in $\{x:|x|\le1\}$, and $\int_{\mathbb R^n}\psi(x)\,dx=0$. The continuous square function $S_{\psi,\beta}(f)$ is defined by (see, for example, \cite{chang} and \cite{chanillo})
\begin{equation*}
S_{\psi,\beta}(f)(x)=\bigg(\iint_{\Gamma_\beta(x)}\big|f*\psi_t(y)\big|^2\frac{dydt}{t^{n+1}}\bigg)^{1/2}.
\end{equation*}

In 2007, Wilson \cite{wilson1} introduced a new square function called intrinsic square function which is universal in a sense (see also \cite{wilson2}). This function is independent of any particular kernel $\psi$, and it dominates pointwise all the above-defined square functions. On the other hand, it is not essentially larger than any particular $S_{\psi,\beta}(f)$. For $0<\alpha\le1$, let ${\mathcal C}_\alpha$ be the family of functions $\varphi$ defined on $\mathbb R^n$ such that $\varphi$ has support containing in $\{x\in\mathbb R^n: |x|\le1\}$, $\int_{\mathbb R^n}\varphi(x)\,dx=0$, and for all $x, x'\in \mathbb R^n$,
\begin{equation*}
|\varphi(x)-\varphi(x')|\le|x-x'|^\alpha.
\end{equation*}
For $(y,t)\in {\mathbb R}^{n+1}_+$ and $f\in L^1_{{loc}}(\mathbb R^n)$, we set
\begin{equation*}
A_\alpha(f)(y,t)=\sup_{\varphi\in{\mathcal C}_\alpha}\big|f*\varphi_t(y)\big|=\sup_{\varphi\in{\mathcal C}_\alpha}\bigg|\int_{\mathbb R^n}\varphi_t(y-z)f(z)\,dz\bigg|.
\end{equation*}
Then we define the intrinsic square function of $f$ (of order $\alpha$) by the formula
\begin{equation*}
\mathcal S_\alpha(f)(x)=\left(\iint_{\Gamma(x)}\Big(A_\alpha(f)(y,t)\Big)^2\frac{dydt}{t^{n+1}}\right)^{1/2}.
\end{equation*}
We can also define varying-aperture versions of $\mathcal S_\alpha(f)$ by the formula
\begin{equation*}
\mathcal S_{\alpha,\beta}(f)(x)=\left(\iint_{\Gamma_\beta(x)}\Big(A_\alpha(f)(y,t)\Big)^2\frac{dydt}{t^{n+1}}\right)^{1/2}.
\end{equation*}
The intrinsic Littlewood-Paley $g$-function and the intrinsic $g^*_\lambda$-function will be given respectively by
\begin{equation*}
g_\alpha(f)(x)=\left(\int_0^\infty\Big(A_\alpha(f)(x,t)\Big)^2\frac{dt}{t}\right)^{1/2}
\end{equation*}
and
\begin{equation*}
g^*_{\lambda,\alpha}(f)(x)=\left(\iint_{{\mathbb R}^{n+1}_+}\left(\frac t{t+|x-y|}\right)^{\lambda n}\Big(A_\alpha(f)(y,t)\Big)^2\frac{dydt}{t^{n+1}}\right)^{1/2}, \quad \lambda>1.
\end{equation*}

In \cite{wilson1} and \cite{wilson2}, Wilson has established the following theorems.

\newtheorem*{thma}{Theorem A}

\begin{thma}
Let $0<\alpha\le1$ and $1<p<\infty$. Then there exists a constant $C>0$ independent of $f$ such that
\begin{equation*}
\|\mathcal S_\alpha(f)\|_{L^p}\le C \|f\|_{L^p}.
\end{equation*}
\end{thma}

\newtheorem*{thmb}{Theorem B}

\begin{thmb}
Let $0<\alpha\le1$ and $p=1$. Then for any $\lambda>0$, there exists a constant $C>0$ independent of $f$ and $\lambda$ such that
\begin{equation*}
\big|\big\{x\in\mathbb R^n:\mathcal S_\alpha(f)(x)>\lambda\big\}\big|\le \frac{C}{\lambda}\int_{\mathbb R^n}|f(x)|\,dx.
\end{equation*}
\end{thmb}

For further discussions about the boundedness of intrinsic square functions on some other function spaces, we refer the reader to \cite{huang,wang4,wang1,wang2,wang3}.

Let $b$ be a locally integrable function on $\mathbb R^n$. In \cite{wang1}, we first introduced the commutators generated by $b$ and intrinsic square functions, which are defined respectively by the following expressions.
\begin{equation*}
\big[b,\mathcal S_\alpha\big](f)(x)=\left(\iint_{\Gamma(x)}\sup_{\varphi\in{\mathcal C}_\alpha}\bigg|\int_{\mathbb R^n}\big[b(x)-b(z)\big]\varphi_t(y-z)f(z)\,dz\bigg|^2\frac{dydt}{t^{n+1}}\right)^{1/2},
\end{equation*}
\begin{equation*}
\big[b,g_\alpha\big](f)(x)=\left(\int_0^\infty\sup_{\varphi\in{\mathcal C}_\alpha}\bigg|\int_{\mathbb R^n}\big[b(x)-b(y)\big]\varphi_t(x-y)f(y)\,dy\bigg|^2\frac{dt}{t}\right)^{1/2},
\end{equation*}
and
\begin{equation*}
\big[b,g^*_{\lambda,\alpha}\big](f)(x)=\left(\iint_{{\mathbb R}^{n+1}_+}\left(\frac t{t+|x-y|}\right)^{\lambda n}\sup_{\varphi\in{\mathcal C}_\alpha}\bigg|\int_{\mathbb R^n}\big[b(x)-b(z)\big]\varphi_t(y-z)f(z)\,dz\bigg|^2\frac{dydt}{t^{n+1}}\right)^{1/2}.
\end{equation*}

On the other hand, the classical Morrey spaces $\mathcal L^{p,\lambda}$ were first introduced by Morrey in \cite{morrey} to study the
local behavior of solutions to second order elliptic partial differential equations. Since then, these spaces play an important role in studying the regularity of solutions to partial differential equations. For the boundedness of the
Hardy-Littlewood maximal operator, the fractional integral operator and the Calder\'on-Zygmund singular integral
operator on these spaces, we refer the reader to \cite{adams,chiarenza,peetre}. For the properties and applications
of classical Morrey spaces, see \cite{fan,fazio1,fazio2} and the references therein.

Let $\Phi=\Phi(r)$, $r>0$, be a growth function, that is, a positive increasing function in $(0,\infty)$ and satisfy the following doubling condition.
\begin{equation}
\Phi(2r)\le D\cdot\Phi(r), \quad \mbox{for all }\,r>0,
\end{equation}
where $D=D(\Phi)\ge1$ is a doubling constant independent of $r$. In \cite{mizuhara}, Mizuhara introduced the following generalized Morrey spaces $L^{p,\Phi}$ and then discussed the boundedness of Hardy-Littlewood maximal operator, the Calder\'on-Zygmund singular integral operator and associated maximal operator on these spaces. For the continuity properties of a class of sublinear operators with rough kernels on $L^{p,\Phi}$, one can see \cite{lu}.

\newtheorem*{defc}{Definition C}

\begin{defc}[\cite{mizuhara}]
Let $1\le p<\infty$. We denote by $L^{p,\Phi}=L^{p,\Phi}(\mathbb R^n)$ the space of all locally integrable functions $f$ defined on $\mathbb R^n$, such that for every $x_0\in\mathbb R^n$ and all $r>0$
\begin{equation}
\int_{B(x_0,r)}|f(x)|^p\,dx\le C^p\Phi(r),
\end{equation}
where $B(x_0,r)=\{x\in\mathbb R^n:|x-x_0|<r\}$ is the ball centered at $x_0$ and with radius $r>0$. Then we let $\|f\|_{L^{p,\Phi}}$ be the smallest constant $C>0$ satisfying (1.2) and $L^{p,\Phi}(\mathbb R^n)$ becomes a Banach space with norm $\|\cdot\|_{L^{p,\Phi}}$.
\end{defc}
Obviously, when $\Phi(r)=r^{\lambda}$ with $0<\lambda<n$, $L^{p,\Phi}$ is just the classical Morrey spaces introduced in \cite{morrey}. We also denote by $WL^{p,\Phi}=WL^{p,\Phi}(\mathbb R^n)$ the generalized weak Morrey spaces of all measurable functions $f$ for which
\begin{equation}
\sup_{\lambda>0}\lambda\cdot\big|\big\{x\in B(x_0,r):|f(x)|>\lambda\big\}\big|\le C\Phi(r),
\end{equation}
for every $x_0\in\mathbb R^n$ and all $r>0$. The smallest constant $C>0$ satisfying (1.3) is also denoted by $\|f\|_{WL^{p,\Phi}}$.

The main purpose of this paper is to discuss the boundedness properties of intrinsic square functions and their commutators with BMO functions on the generalized Morrey spaces $L^{p,\Phi}(\mathbb R^n)$ for all $1\le p<\infty$. Our main results in the paper are formulated as follows.

\newtheorem{theorem}{Theorem}[section]

\begin{theorem}
Let $0<\alpha\le1$ and $1<p<\infty$. Assume that $\Phi$ satisfies $(1.1)$ and $1\le D(\Phi)<2^n$, then there is a
constant $C>0$ independent of $f$ such that
\begin{equation*}
\big\|\mathcal S_\alpha(f)\big\|_{L^{p,\Phi}}\le C\|f\|_{L^{p,\Phi}}.
\end{equation*}
\end{theorem}

\begin{theorem}
Let $0<\alpha\le1$ and $p=1$. Assume that $\Phi$ satisfies $(1.1)$ and $1\le D(\Phi)<2^n$, then there is a
constant $C>0$ independent of $f$ such that
\begin{equation*}
\big\|\mathcal S_\alpha(f)\big\|_{WL^{1,\Phi}}\le C\|f\|_{L^{1,\Phi}}.
\end{equation*}
\end{theorem}

\begin{theorem}
Let $0<\alpha\le1$, $1<p<\infty$ and $b\in BMO(\mathbb R^n)$. Assume that $\Phi$ satisfies $(1.1)$ and $1\le D(\Phi)<2^n$, then there is a
constant $C>0$ independent of $f$ such that
\begin{equation*}
\big\|\big[b,\mathcal S_\alpha\big](f)\big\|_{L^{p,\Phi}}\le C\|f\|_{L^{p,\Phi}}.
\end{equation*}
\end{theorem}

\begin{theorem}
Let $0<\alpha\le1$ and $1<p<\infty$. Suppose that $\lambda>3$, $\Phi$ satisfies $(1.1)$ and $1\le D(\Phi)<2^n$, then there is a
constant $C>0$ independent of $f$ such that
\begin{equation*}
\big\|g^*_{\lambda,\alpha}(f)\big\|_{L^{p,\Phi}}\le C\|f\|_{L^{p,\Phi}}.
\end{equation*}
\end{theorem}

\begin{theorem}
Let $0<\alpha\le1$ and $p=1$. Suppose that $\lambda>{(3n+2\alpha)}/n$, $\Phi$ satisfies $(1.1)$ and $1\le D(\Phi)<2^n$, then there is a
constant $C>0$ independent of $f$ such that
\begin{equation*}
\big\|g^*_{\lambda,\alpha}(f)\big\|_{WL^{1,\Phi}}\le C\|f\|_{L^{1,\Phi}}.
\end{equation*}
\end{theorem}

\begin{theorem}
Let $0<\alpha\le1$, $1<p<\infty$ and $b\in BMO(\mathbb R^n)$. Suppose that $\lambda>3$, $\Phi$ satisfies $(1.1)$ and $1\le D(\Phi)<2^n$, then there is a
constant $C>0$ independent of $f$ such that
\begin{equation*}
\big\|\big[b,g^*_{\lambda,\alpha}\big](f)\big\|_{L^{p,\Phi}}\le C\|f\|_{L^{p,\Phi}}.
\end{equation*}
\end{theorem}

In \cite{wilson1}, Wilson also showed that for any $0<\alpha\le1$, the functions $\mathcal S_\alpha(f)(x)$ and $g_\alpha(f)(x)$ are pointwise comparable. Thus, as a direct consequence of Theorems 1.1, 1.2 and 1.3, we obtain the following

\newtheorem{corollary}[theorem]{Corollary}

\begin{corollary}
Let $0<\alpha\le1$ and $1<p<\infty$. Assume that $\Phi$ satisfies $(1.1)$ and $1\le D(\Phi)<2^n$, then there is a
constant $C>0$ independent of $f$ such that
\begin{equation*}
\big\|g_\alpha(f)\big\|_{L^{p,\Phi}}\le C\|f\|_{L^{p,\Phi}}.
\end{equation*}
\end{corollary}

\begin{corollary}
Let $0<\alpha\le1$ and $p=1$. Assume that $\Phi$ satisfies $(1.1)$ and $1\le D(\Phi)<2^n$, then there is a
constant $C>0$ independent of $f$ such that
\begin{equation*}
\big\|g_\alpha(f)\big\|_{WL^{1,\Phi}}\le C\|f\|_{L^{1,\Phi}}.
\end{equation*}
\end{corollary}

\begin{corollary}
Let $0<\alpha\le1$, $1<p<\infty$ and $b\in BMO(\mathbb R^n)$. Suppose that $\Phi$ satisfies $(1.1)$ and $1\le D(\Phi)<2^n$, then there is a
constant $C>0$ independent of $f$ such that
\begin{equation*}
\big\|\big[b,g_\alpha\big](f)\big\|_{L^{p,\Phi}}\le C\|f\|_{L^{p,\Phi}}.
\end{equation*}
\end{corollary}

Throughout this article, $B=B(x_0,r)$ denotes the ball with the center $x_0$ and radius $r$. Given a ball $B$ and $\lambda>0$, $\lambda B$ denotes the ball with the same center as $B$ whose radius is $\lambda$ times that of $B$. For any measurable set $E$ in $\mathbb R^n$, we also denote the Lebesgue measure of $E$ by $|E|$. Moreover, $C$ always denote a positive constant independent of the main parameters involved, but it may be different from line to line.

\section{Proofs of Theorems 1.1 and 1.2}

\begin{proof}[Proof of Theorem 1.1]
Let $f\in L^{p,\Phi}$ with $1<p<\infty$. For any ball $B=B(x_0,r)\subseteq\mathbb R^n$ with $x_0\in\mathbb R^n$ and $r>0$, we write $f=f_1+f_2$, where $f_1=f\chi_{_{2B}}$, $\chi_{_{2B}}$ denotes the characteristic function of $2B=B(x_0,2r)$. Since $\mathcal S_\alpha$($0<\alpha\le1$) is a sublinear operator, then we have
\begin{equation*}
\begin{split}
&\frac{1}{\Phi(r)^{1/p}}\bigg(\int_{B(x_0,r)}|\mathcal S_\alpha(f)(x)|^p\,dx\bigg)^{1/p}\\
\le\,&\frac{1}{\Phi(r)^{1/p}}\bigg(\int_{B(x_0,r)}|\mathcal S_\alpha(f_1)(x)|^p\,dx\bigg)^{1/p}+
\frac{1}{\Phi(r)^{1/p}}\bigg(\int_{B(x_0,r)}|\mathcal S_\alpha(f_2)(x)|^p\,dx\bigg)^{1/p}\\
=\,&I_1+I_2.
\end{split}
\end{equation*}
For the term $I_1$, by Theorem A and the condition (1.1), we obtain
\begin{equation*}
\begin{split}
I_1&\le C\cdot\frac{1}{\Phi(r)^{1/p}}\bigg(\int_{2B}|f(x)|^p\,dx\bigg)^{1/p}\\
&\le C\|f\|_{L^{p,\Phi}}\cdot\frac{\Phi(2r)^{1/p}}{\Phi(r)^{1/p}}\\
&\le C\|f\|_{L^{p,\Phi}}.
\end{split}
\end{equation*}
We now turn to estimate the other term $I_2$. For any $\varphi\in{\mathcal C}_\alpha$, $0<\alpha\le1$ and $(y,t)\in\Gamma(x)$, we have
\begin{align}
\big|f_2*\varphi_t(y)\big|&=\bigg|\int_{(2B)^c}\varphi_t(y-z)f(z)\,dz\bigg|\notag\\
&\le C\cdot t^{-n}\int_{(2B)^c\cap\{z:|y-z|\le t\}}|f(z)|\,dz\notag\\
&\le C\cdot t^{-n}\sum_{k=1}^\infty\int_{(2^{k+1}B\backslash 2^{k}B)\cap\{z:|y-z|\le t\}}|f(z)|\,dz.
\end{align}
For any $x\in B$, $(y,t)\in\Gamma(x)$ and $z\in\big(2^{k+1}B\backslash 2^{k}B\big)\cap B(y,t)$, then by a direct computation, we can easily see that
\begin{equation*}
2t\ge |x-y|+|y-z|\ge|x-z|\ge|z-x_0|-|x-x_0|\ge 2^{k-1}r.
\end{equation*}
Thus, by using the above inequality (2.1) and Minkowski's integral inequality, we deduce
\begin{align}
\big|\mathcal S_\alpha(f_2)(x)\big|&=\left(\iint_{\Gamma(x)}\sup_{\varphi\in{\mathcal C}_\alpha}|f_2*\varphi_t(y)|^2\frac{dydt}{t^{n+1}}\right)^{1/2}\notag\\
&\le C\left(\int_{2^{k-2}r}^\infty\int_{|x-y|<t}\bigg|t^{-n}\sum_{k=1}^\infty\int_{2^{k+1}B\backslash 2^{k}B}|f(z)|\,dz\bigg|^2\frac{dydt}{t^{n+1}}\right)^{1/2}\notag\\
&\le C\bigg(\sum_{k=1}^\infty\int_{2^{k+1}B\backslash 2^{k}B}|f(z)|\,dz\bigg)\bigg(\int_{2^{k-2}r}^\infty\frac{dt}{t^{2n+1}}\bigg)^{1/2}\notag\\
&\le C\sum_{k=1}^\infty\frac{1}{|B(x_0,2^{k+1}r)|}\int_{2^{k+1}B\backslash 2^{k}B}|f(z)|\,dz.
\end{align}
An application of H\"older's inequality leads to that
\begin{align}
\frac{1}{|B(x_0,2^{k+1}r)|}\int_{2^{k+1}B\backslash 2^{k}B}|f(z)|\,dz&\le\frac{1}{|B(x_0,2^{k+1}r)|^{1/p}}\bigg(\int_{2^{k+1}B}|f(z)|^p\,dz\bigg)^{1/p}\notag\\
&\le C\|f\|_{L^{p,\Phi}}\cdot\frac{\Phi(2^{k+1}r)^{1/p}}{|B(x_0,2^{k+1}r)|^{1/p}}.
\end{align}
Hence, substituting the above inequality (2.3) into (2.2), we have that for all $x\in B=B(x_0,r)$,
\begin{equation}
\big|\mathcal S_\alpha(f_2)(x)\big|\le C\|f\|_{L^{p,\Phi}}\sum_{k=1}^\infty\frac{\Phi(2^{k+1}r)^{1/p}}{|B(x_0,2^{k+1}r)|^{1/p}},
\end{equation}
which implies
\begin{equation*}
I_2\le C\|f\|_{L^{p,\Phi}}\sum_{k=1}^\infty\frac{|B(x_0,r)|^{1/p}}{\Phi(r)^{1/p}}
\cdot\frac{\Phi(2^{k+1}r)^{1/p}}{|B(x_0,2^{k+1}r)|^{1/p}}.
\end{equation*}
Since $1\le D(\Phi)<2^n$, then by using the doubling condition (1.1) of $\Phi$, we know that the above series is bounded by an absolute constant.
\begin{align}
\sum_{k=1}^\infty\frac{|B(x_0,r)|^{1/p}}{\Phi(r)^{1/p}}\cdot\frac{\Phi(2^{k+1}r)^{1/p}}{|B(x_0,2^{k+1}r)|^{1/p}}
&\le C\sum_{k=1}^\infty\left(\frac{D(\Phi)}{2^{n}}\right)^{{(k+1)}/p}\notag\\
&\le C.
\end{align}
Therefore
\begin{equation*}
I_2\le C\|f\|_{L^{p,\Phi}}.
\end{equation*}
Combining the above estimates for $I_1$ and $I_2$ and then taking the supremum over all balls $B=B(x_0,r)\subseteq\mathbb R^n$, we complete the proof of Theorem 1.1.
\end{proof}

\begin{proof}[Proof of Theorem 1.2]
Let $f\in L^{1,\Phi}$. Fix a ball $B=B(x_0,r)\subseteq\mathbb R^n$ and decompose $f=f_1+f_2$, where $f_1=f\chi_{_{2B}}$. For any given $\lambda>0$, we write
\begin{equation*}
\begin{split}
&\big|\big\{x\in B(x_0,r):|\mathcal S_\alpha(f)(x)|>\lambda\big\}\big|\\
\le\,& \big|\big\{x\in B(x_0,r):|\mathcal S_\alpha(f_1)(x)|>\lambda/2\big\}\big|+\big|\big\{x\in B(x_0,r):|\mathcal S_\alpha(f_2)(x)|>\lambda/2\big\}\big|\\
  =\,&J_1+J_2.
\end{split}
\end{equation*}
Theorem B and the condition (1.1) imply
\begin{equation*}
\begin{split}
J_1&\le\frac{C}{\lambda}\int_{2B}|f(y)|\,dy\\
&\le\frac{C\cdot\Phi(2r)}{\lambda}\|f\|_{L^{1,\Phi}}\\
&\le\frac{C\cdot\Phi(r)}{\lambda}\|f\|_{L^{1,\Phi}}.
\end{split}
\end{equation*}
We turn our attention to the estimate of $J_2$. Using the preceding estimate (2.2), we can deduce that for all $x\in B(x_0,r)$,
\begin{equation*}
\begin{split}
\big|\mathcal S_\alpha(f_2)(x)\big|&\le C\sum_{k=1}^\infty\frac{1}{|B(x_0,2^{k+1}r)|}\int_{2^{k+1}B\backslash 2^{k}B}|f(z)|\,dz\\
&\le C\|f\|_{L^{1,\Phi}}\sum_{k=1}^\infty\frac{\Phi(2^{k+1}r)}{|B(x_0,2^{k+1}r)|}\\
&=C\|f\|_{L^{1,\Phi}}\cdot\frac{\Phi(r)}{|B(x_0,r)|}\sum_{k=1}^\infty
\frac{|B(x_0,r)|}{\Phi(r)}\cdot\frac{\Phi(2^{k+1}r)}{|B(x_0,2^{k+1}r)|}.
\end{split}
\end{equation*}
Note that $1\le D(\Phi)<2^n$. Arguing as in the proof of (2.5), we can get
\begin{align}
\sum_{k=1}^\infty\frac{|B(x_0,r)|}{\Phi(r)}\cdot\frac{\Phi(2^{k+1}r)}{|B(x_0,2^{k+1}r)|}&\le
\sum_{k=1}^\infty\left(\frac{D(\Phi)}{2^n}\right)^{k+1}\notag\\
&\le C.
\end{align}
Hence
\begin{align}
\big|\mathcal S_\alpha(f_2)(x)\big|&\le C\|f\|_{L^{1,\Phi}}\cdot\frac{\Phi(r)}{|B(x_0,r)|}.
\end{align}
If $\big\{x\in B(x_0,r):|\mathcal S_\alpha(f_2)(x)|>\lambda/2\big\}=\O$, then the inequality
\begin{equation*}
J_2\le\frac{C\cdot\Phi(r)}{\lambda}\|f\|_{L^{1,\Phi}}
\end{equation*}
holds trivially. Now we may suppose that $\big\{x\in B(x_0,r):|\mathcal S_\alpha(f_2)(x)|>\lambda/2\big\}\neq\O$,
then by the inequality (2.7), we can see that
\begin{equation*}
\lambda\le C\|f\|_{L^{1,\Phi}}\cdot\frac{\Phi(r)}{|B(x_0,r)|},
\end{equation*}
which is equivalent to
\begin{equation*}
|B(x_0,r)|\le \frac{C\cdot\Phi(r)}{\lambda}\|f\|_{L^{1,\Phi}}.
\end{equation*}
Therefore
\begin{equation*}
J_2\le|B(x_0,r)|\le \frac{C\cdot\Phi(r)}{\lambda}\|f\|_{L^{1,\Phi}}.
\end{equation*}
Summing up the above estimates for $J_1$ and $J_2$, and then taking the supremum over all balls $B=B(x_0,r)\subseteq\mathbb R^n$ and all $\lambda>0$, we finish the proof of Theorem 1.2.
\end{proof}

\section{Proof of Theorem 1.3}

Before proving the main theorem in this section, let us first recall the definition of the space of $BMO(\mathbb R^n)$ (Bounded Mean Oscillation). A locally integrable function $b$ is said to be in $BMO(\mathbb R^n)$ if
\begin{equation*}
\|b\|_*=\sup_{B}\frac{1}{|B|}\int_B|b(x)-b_B|\,dx<\infty,
\end{equation*}
where $b_B$ stands for the average of $b$ on $B$, i.e., $b_B=\frac{1}{|B|}\int_B b(y)\,dy$ and the supremum is taken
over all balls $B$ in $\mathbb R^n$. Modulo constants, the space $BMO(\mathbb R^n)$ is a Banach space with respect to the norm $\|\cdot\|_*$.

\begin{theorem}[\cite{duoand,john}]
Assume that $b\in BMO(\mathbb R^n)$. Then for any $1\le p<\infty$, we have
\begin{equation*}
\sup_B\bigg(\frac{1}{|B|}\int_B\big|b(x)-b_B\big|^p\,dx\bigg)^{1/p}\le C\|b\|_*.
\end{equation*}
\end{theorem}

Given a real-valued function $b\in BMO(\mathbb R^n)$, we shall follow the idea developed in \cite{alvarez,ding} and denote $F(\xi)=e^{\xi[b(x)-b(z)]}$, $\xi\in\mathbb C$. Then by the analyticity of $F(\xi)$ on $\mathbb C$ and the Cauchy integral formula, we get
\begin{equation*}
\begin{split}
b(x)-b(z)=F'(0)&=\frac{1}{2\pi i}\int_{|\xi|=1}\frac{F(\xi)}{\xi^2}\,d\xi\\
&=\frac{1}{2\pi}\int_0^{2\pi}e^{e^{i\theta}[b(x)-b(z)]}e^{-i\theta}\,d\theta.
\end{split}
\end{equation*}
Thus, for any $\varphi\in{\mathcal C}_\alpha$, $0<\alpha\le1$, we obtain
\begin{align}
\bigg|\int_{\mathbb R^n}\big[b(x)-b(z)\big]\varphi_t(y-z)f(z)\,dz\bigg|&=
\bigg|\frac{1}{2\pi}\int_0^{2\pi}\bigg(\int_{\mathbb R^n}\varphi_t(y-z)e^{-e^{i\theta}b(z)}f(z)\,dz\bigg)
e^{e^{i\theta}b(x)}e^{-i\theta}\,d\theta\bigg|\notag\\
&\le\frac{1}{2\pi}\int_0^{2\pi}\sup_{\varphi\in{\mathcal C}_\alpha}\bigg|\int_{\mathbb R^n}\varphi_t(y-z)e^{-e^{i\theta}b(z)}f(z)\,dz\bigg|e^{\cos\theta\cdot b(x)}\,d\theta\notag\\
&\le\frac{1}{2\pi}\int_0^{2\pi}A_\alpha\big(e^{-e^{i\theta}b}\cdot f\big)(y,t)\cdot e^{\cos\theta\cdot b(x)}\,d\theta.
\end{align}
So we have
\begin{equation*}
\big|\big[b,\mathcal S_\alpha\big](f)(x)\big|\le\frac{1}{2\pi}\int_0^{2\pi}
\mathcal S_\alpha\big(e^{-e^{i\theta}b}\cdot f\big)(x)\cdot e^{\cos\theta\cdot b(x)}\,d\theta.
\end{equation*}
Then in view of Theorem A, by using the same arguments as in \cite{ding}, we can also show the following (see \cite{wang1} for the weighted case).

\begin{theorem}
Let $0<\alpha\le1$ and $1<p<\infty$. Then the commutator $\big[b,\mathcal S_\alpha\big]$ is bounded from $L^p(\mathbb R^n)$ into itself whenever $b\in BMO(\mathbb R^n)$.
\end{theorem}

\begin{proof}[Proof of Theorem 1.3]
Let $f\in L^{p,\Phi}$ with $1<p<\infty$. For each fixed ball $B=B(x_0,r)\subseteq\mathbb R^n$, we let $f=f_1+f_2$, where $f_1=f\chi_{_{2B}}$. Then we can write
\begin{equation*}
\begin{split}
&\frac{1}{\Phi(r)^{1/p}}\bigg(\int_{B(x_0,r)}\big|\big[b,\mathcal S_\alpha\big](f)(x)\big|^p\,dx\bigg)^{1/p}\\
\le\,&\frac{1}{\Phi(r)^{1/p}}\bigg(\int_{B(x_0,r)}\big|\big[b,\mathcal S_\alpha\big](f_1)(x)\big|^p\,dx\bigg)^{1/p}+
\frac{1}{\Phi(r)^{1/p}}\bigg(\int_{B(x_0,r)}\big|\big[b,\mathcal S_\alpha\big](f_2)(x)\big|^p\,dx\bigg)^{1/p}\\
=\,&K_1+K_2.
\end{split}
\end{equation*}
Applying Theorem 3.2 and the condition (1.1), we thus obtain
\begin{align}
K_1&\le C\|b\|_*\cdot\frac{1}{\Phi(r)^{1/p}}\bigg(\int_{2B}|f(x)|^p\,dx\bigg)^{1/p}\notag\\
&\le C\|b\|_*\|f\|_{L^{p,\Phi}}\cdot\frac{\Phi(2r)^{1/p}}{\Phi(r)^{1/p}}\notag\\
&\le C\|b\|_*\|f\|_{L^{p,\Phi}}.
\end{align}
We now turn to deal with the term $K_2$. For any given $x\in B(x_0,r)$ and $(y,t)\in\Gamma(x)$, we have
\begin{equation*}
\begin{split}
\sup_{\varphi\in{\mathcal C}_\alpha}\bigg|\int_{\mathbb R^n}\big[b(x)-b(z)\big]\varphi_t(y-z)f_2(z)\,dz\bigg|&\le
\big|b(x)-b_B\big|\cdot\sup_{\varphi\in{\mathcal C}_\alpha}\bigg|\int_{\mathbb R^n}\varphi_t(y-z)f_2(z)\,dz\bigg|\\
&+\sup_{\varphi\in{\mathcal C}_\alpha}\bigg|\int_{\mathbb R^n}\big[b(z)-b_B\big]\varphi_t(y-z)f_2(z)\,dz\bigg|
\end{split}
\end{equation*}
Hence
\begin{equation*}
\begin{split}
\big|\big[b,\mathcal S_\alpha\big](f_2)(x)\big|&\le\big|b(x)-b_B\big|\cdot \mathcal S_\alpha(f_2)(x)\\
&+\left(\iint_{\Gamma(x)}\sup_{\varphi\in{\mathcal C}_\alpha}\bigg|\int_{\mathbb R^n}\big[b(z)-b_B\big]\varphi_t(y-z)f_2(z)\,dz\bigg|^2\frac{dydt}{t^{n+1}}\right)^{1/2}\\
&=\mbox{\upshape I+II}.
\end{split}
\end{equation*}
In the proof of Theorem 1.1, we have already proved that for any $x\in B(x_0,r)$,
\begin{equation}
\big|\mathcal S_\alpha(f_2)(x)\big|\le C\|f\|_{L^{p,\Phi}}\sum_{k=1}^\infty\frac{\Phi(2^{k+1}r)^{1/p}}{|B(x_0,2^{k+1}r)|^{1/p}},
\end{equation}
From the inequalities (2.5), (3.3) and Theorem 3.1, it follows that
\begin{align*}
\frac{1}{\Phi(r)^{1/p}}\bigg(\int_B \mbox{\upshape I}^p\,dx\bigg)^{1/p}&\le C\|f\|_{L^{p,\Phi}}\cdot\frac{1}{\Phi(r)^{1/p}}\sum_{k=1}^\infty\frac{\Phi(2^{k+1}r)^{1/p}}{|B(x_0,2^{k+1}r)|^{1/p}}
\cdot\bigg(\int_B\big|b(x)-b_B\big|^p\,dx\bigg)^{1/p}\notag\\
&\le C\|b\|_*\|f\|_{L^{p,\Phi}}\sum_{k=1}^\infty\frac{|B(x_0,r)|^{1/p}}{\Phi(r)^{1/p}}
\cdot\frac{\Phi(2^{k+1}r)^{1/p}}{|B(x_0,2^{k+1}r)|^{1/p}}\notag\\
\end{align*}
\begin{align}
&\le C\|b\|_*\|f\|_{L^{p,\Phi}}\sum_{k=1}^\infty\left(\frac{D(\Phi)}{2^{n}}\right)^{{(k+1)}/p}\notag\\
&\le C\|b\|_*\|f\|_{L^{p,\Phi}}.
\end{align}
On the other hand
\begin{equation*}
\begin{split}
\mbox{\upshape II}&=\left(\iint_{\Gamma(x)}\sup_{\varphi\in{\mathcal C}_\alpha}\bigg|\int_{(2B)^c}\big[b(z)-b_B\big]\varphi_t(y-z)f(z)\,dz\bigg|^2\frac{dydt}{t^{n+1}}\right)^{1/2}\\
&\le C\left(\iint_{\Gamma(x)}
\bigg|t^{-n}\sum_{k=1}^\infty\int_{(2^{k+1}B\backslash 2^{k}B)\cap\{z:|y-z|\le t\}}|b(z)-b_B||f(z)|\,dz\bigg|^2\frac{dydt}{t^{n+1}}\right)^{1/2}\\
&\le C\left(\iint_{\Gamma(x)}
\bigg|t^{-n}\sum_{k=1}^\infty\int_{(2^{k+1}B\backslash 2^{k}B)\cap\{z:|y-z|\le t\}}\big|b(z)-b_{2^{k+1}B}\big||f(z)|\,dz\bigg|^2\frac{dydt}{t^{n+1}}\right)^{1/2}\\
&+C\left(\iint_{\Gamma(x)}
\bigg|t^{-n}\sum_{k=1}^\infty\big|b_{2^{k+1}B}-b_B\big|\cdot\int_{(2^{k+1}B\backslash 2^{k}B)\cap\{z:|y-z|\le t\}}|f(z)|\,dz\bigg|^2\frac{dydt}{t^{n+1}}\right)^{1/2}\\
&= \mbox{\upshape III+IV}.
\end{split}
\end{equation*}
We denote the conjugate exponent of $p>1$ by $p'=p/{(p-1)}$. Then by H\"older's inequality and Theorem 3.1, we obtain
\begin{align}
&\int_{2^{k+1}B\backslash 2^{k}B}\big|b(z)-b_{2^{k+1}B}\big||f(z)|\,dz\notag\\
\le&\,\bigg(\int_{2^{k+1}B}\big|b(z)-b_{2^{k+1}B}\big|^{p'}\,dz\bigg)^{1/{p'}}
\bigg(\int_{2^{k+1}B}\big|f(z)\big|^p\,dz\bigg)^{1/p}\notag\\
\le&\,C\|b\|_*\|f\|_{L^{p,\Phi}}\cdot \big|2^{k+1}B\big|^{1/{p'}}\Phi(2^{k+1}r)^{1/p}.
\end{align}
In addition, we note that in this case, $t\ge2^{k-2}r$ as in Theorem 1.1. Then it follows from Minkowski's integral inequality and the above inequality (3.5) that
\begin{equation*}
\begin{split}
\mbox{\upshape III}&\le C\left(\int_{2^{k-2}r}^\infty\int_{|x-y|<t}\bigg|t^{-n}\sum_{k=1}^\infty\int_{2^{k+1}B\backslash 2^{k}B}\big|b(z)-b_{2^{k+1}B}\big||f(z)|\,dz\bigg|^2\frac{dydt}{t^{n+1}}\right)^{1/2}\\
&\le C\bigg(\sum_{k=1}^\infty\int_{2^{k+1}B\backslash 2^{k}B}\big|b(z)-b_{2^{k+1}B}\big||f(z)|\,dz\bigg)\bigg(\int_{2^{k-2}r}^\infty\frac{dt}{t^{2n+1}}\bigg)^{1/2}\\
&\le C\|b\|_*\|f\|_{L^{p,\Phi}}\cdot \sum_{k=1}^\infty\frac{\Phi(2^{k+1}r)^{1/p}}{|B(x_0,2^{k+1}r)|^{1/p}}.
\end{split}
\end{equation*}
Hence, it follows directly from the inequality (2.5) that
\begin{align}
\frac{1}{\Phi(r)^{1/p}}\bigg(\int_B \mbox{\upshape III}^p\,dx\bigg)^{1/p}&\le C\|b\|_*\|f\|_{L^{p,\Phi}}\sum_{k=1}^\infty\frac{|B(x_0,r)|^{1/p}}{\Phi(r)^{1/p}}
\cdot\frac{\Phi(2^{k+1}r)^{1/p}}{|B(x_0,2^{k+1}r)|^{1/p}}\notag\\
&\le C\|b\|_*\|f\|_{L^{p,\Phi}}\sum_{k=1}^\infty\left(\frac{D(\Phi)}{2^{n}}\right)^{{(k+1)}/p}\notag\\
&\le C\|b\|_*\|f\|_{L^{p,\Phi}}.
\end{align}
Now let us deal with the last term \mbox{\upshape IV}. Since $b\in BMO(\mathbb R^n)$, then a trivial calculation shows that
\begin{equation}
\big|b_{2^{k+1}B}-b_B\big|\le C\cdot(k+1)\|b\|_*.
\end{equation}
Thus, by using Minkowski's integral inequality and the inequalities (2.3) and (3.7), we have
\begin{equation*}
\begin{split}
\mbox{\upshape IV}&\le C\left(\int_{2^{k-2}r}^\infty\int_{|x-y|<t}\bigg|t^{-n}\sum_{k=1}^\infty\big|b_{2^{k+1}B}-b_B\big|
\cdot\int_{2^{k+1}B\backslash 2^{k}B}|f(z)|\,dz\bigg|^2\frac{dydt}{t^{n+1}}\right)^{1/2}\\
&\le C\|b\|_*\bigg(\sum_{k=1}^\infty(k+1)\cdot\int_{2^{k+1}B\backslash 2^{k}B}|f(z)|\,dz\bigg)\bigg(\int_{2^{k-2}r}^\infty\frac{dt}{t^{2n+1}}\bigg)^{1/2}\\
&\le C\|b\|_*\|f\|_{L^{p,\Phi}}\sum_{k=1}^\infty(k+1)\cdot\frac{\Phi(2^{k+1}r)^{1/p}}{|B(x_0,2^{k+1}r)|^{1/p}}.
\end{split}
\end{equation*}
Therefore
\begin{align}
\frac{1}{\Phi(r)^{1/p}}\bigg(\int_B \mbox{\upshape IV}^p\,dx\bigg)^{1/p}&\le C\|b\|_*\|f\|_{L^{p,\Phi}}\sum_{k=1}^\infty (k+1)\cdot\frac{|B(x_0,r)|^{1/p}}{\Phi(r)^{1/p}}
\cdot\frac{\Phi(2^{k+1}r)^{1/p}}{|B(x_0,2^{k+1}r)|^{1/p}}\notag\\
&\le C\|b\|_*\|f\|_{L^{p,\Phi}}\sum_{k=1}^\infty (k+1)\cdot\left(\frac{D(\Phi)}{2^{n}}\right)^{{(k+1)}/p}\notag\\
&\le C\|b\|_*\|f\|_{L^{p,\Phi}},
\end{align}
where we have used the inequality (2.5). Summarizing the estimates (3.6) and (3.8) derived above, we thus obtain
\begin{equation}
\frac{1}{\Phi(r)^{1/p}}\bigg(\int_B \mbox{\upshape II}^p\,dx\bigg)^{1/p}\le C\|b\|_*\|f\|_{L^{p,\Phi}}.
\end{equation}
Combining the inequalities (3.2), (3.4) with the above inequality (3.9) and then taking the supremum over all balls $B=B(x_0,r)\subseteq\mathbb R^n$, we complete the proof of Theorem 1.3.
\end{proof}

\section{Proof of Theorem 1.4}

In order to prove the main theorem of this section, we need to establish the following three lemmas. Actually, these results are essentially contained in \cite{torchinsky}. For the sake of completeness, we give its proofs here (see also \cite{wang1} for the weighted case).

\newtheorem{lemma}[theorem]{Lemma}

\begin{lemma}
Let $0<\alpha\le1$ and $p=2$. Then for any $j\in\mathbb Z_+$, we have
\begin{equation*}
\big\|\mathcal S_{\alpha,2^j}(f)\big\|_{L^2}\le C\cdot2^{{jn}/2}\big\|\mathcal S_\alpha(f)\big\|_{L^2}.
\end{equation*}
\end{lemma}

\begin{proof}
For every $j\in\mathbb Z_+$, by the definition of $\mathcal S_{\alpha,2^j}$, we obtain
\begin{equation*}
\begin{split}
\big\|\mathcal S_{\alpha,2^j}(f)\big\|_{L^2}^2&=\int_{\mathbb R^n}\bigg(\iint_{{\mathbb R}^{n+1}_+}\Big(A_\alpha(f)(y,t)\Big)^2\chi_{|x-y|<2^j t}\frac{dydt}{t^{n+1}}\bigg)\,dx\\
&=\iint_{{\mathbb R}^{n+1}_+}\bigg(\int_{|x-y|<2^j t}\,dx\bigg)\Big(A_\alpha(f)(y,t)\Big)^2\frac{dydt}{t^{n+1}}\\
&\le C\cdot2^{jn}\iint_{{\mathbb R}^{n+1}_+}\bigg(\int_{|x-y|<t}\,dx\bigg)\Big(A_\alpha(f)(y,t)\Big)^2\frac{dydt}{t^{n+1}}\\
&=C\cdot 2^{jn}\big\|\mathcal S_\alpha(f)\big\|_{L^2}^2.
\end{split}
\end{equation*}
Taking square-roots on both sides of the above inequality, we are done.
\end{proof}

\begin{lemma}
Let $0<\alpha\le1$ and $2<p<\infty$. Then for any $j\in\mathbb Z_+$, we have
\begin{equation*}
\big\|\mathcal S_{\alpha,2^j}(f)\big\|_{L^p}\le C\cdot2^{{jn}/2}\big\|\mathcal S_\alpha(f)\big\|_{L^p}.
\end{equation*}
\end{lemma}

\begin{proof}
For any $j\in\mathbb Z_+$, it is easy to see that
\begin{equation*}
\big\|\mathcal S_{\alpha,2^j}(f)\big\|^2_{L^p}=\big\|\mathcal S_{\alpha,2^j}(f)^2\big\|_{L^{p/2}}.
\end{equation*}
Since $p/2>1$, then by duality, we have
\begin{align}
&\big\|\mathcal S_{\alpha,2^j}(f)^2\big\|_{L^{p/2}}\notag\\
=&\underset{\|g\|_{L^{{(p/2)}'}}\le1}{\sup}\left|\int_{\mathbb R^n}\mathcal S_{\alpha,2^j}(f)(x)^2g(x)\,dx\right|\notag\\
=&\underset{\|g\|_{L^{{(p/2)}'}}\le1}{\sup}\left|\int_{\mathbb R^n}\bigg(\iint_{{\mathbb R}^{n+1}_+}\Big(A_\alpha(f)(y,t)\Big)^2\chi_{|x-y|<2^j t}\frac{dydt}{t^{n+1}}\bigg)g(x)\,dx\right|\notag\\
=&\underset{\|g\|_{L^{{(p/2)}'}}\le1}{\sup}\left|\iint_{{\mathbb R}^{n+1}_+}\bigg(\int_{|x-y|<2^jt}g(x)\,dx\bigg)\Big(A_\alpha(f)(y,t)\Big)^2 \frac{dydt}{t^{n+1}}\right|.
\end{align}
Recall that the Hardy-Littlewood maximal operator $M$ is defined by
\begin{equation*}
M(f)(x)=\underset{x\in B}{\sup}\frac{1}{|B|}\int_B|f(y)|\,dy,
\end{equation*}
where the supremum is taken over all balls $B$ which contain $x$. Then we get
\begin{align}
\int_{|x-y|<2^jt}g(x)\,dx&\le 2^{jn}|B(y,t)|\cdot\frac{1}{|B(y,2^jt)|}\int_{B(y,2^jt)}g(x)\,dx\notag\\
&\le 2^{jn}|B(y,t)|\underset{x\in B(y,t)}{\inf}M(g)(x)\notag\\
&\le 2^{jn}\int_{|x-y|<t}M(g)(x)\,dx.
\end{align}
Substituting the above inequality (4.2) into (4.1) and using H\"older's inequality together with the $L^{(p/2)'}$ boundedness of $M$, we thus obtain
\begin{equation*}
\begin{split}
\big\|\mathcal S_{\alpha,2^j}(f)^2\big\|_{L^{p/2}}&\le 2^{jn}\underset{\|g\|_{L^{{(p/2)}'}}\le1}{\sup}
\left|\int_{\mathbb R^n}\mathcal S_\alpha(f)(x)^2M(g)(x)\,dx\right|\\
&\le 2^{jn}\big\|\mathcal S_\alpha(f)^2\big\|_{L^{p/2}}\underset{\|g\|_{L^{{(p/2)}'}}\le1}
{\sup}\big\|M(g)\big\|_{L^{{(p/2)}'}}\\
&\le C\cdot2^{jn}\big\|\mathcal S_\alpha(f)^2\big\|_{L^{p/2}}\\
&= C\cdot2^{jn}\big\|\mathcal S_\alpha(f)\big\|^2_{L^p}.
\end{split}
\end{equation*}
This implies the desired result.
\end{proof}

\begin{lemma}
Let $0<\alpha\le1$ and $1\le p<2$. Then for any $j\in\mathbb Z_+$, we have
\begin{equation*}
\big\|\mathcal S_{\alpha,2^j}(f)\big\|_{L^p}\le C\cdot2^{{jn}/p}\big\|\mathcal S_\alpha(f)\big\|_{L^p}.
\end{equation*}
\end{lemma}

\begin{proof}
We will adopt the same method as in \cite{torchinsky}. For any $j\in\mathbb Z_+$, set
$\Omega_\lambda=\big\{x\in\mathbb R^n:\mathcal S_\alpha(f)(x)>\lambda\big\}$ and $\Omega_{\lambda,j}=\big\{x\in\mathbb R^n:\mathcal S_{\alpha,2^j}(f)(x)>\lambda\big\}.$ We also set
\begin{equation*}
\Omega^*_\lambda=\Big\{x\in\mathbb R^n:M(\chi_{\Omega_\lambda})(x)>\frac{1}{2^{(jn+1)}}\Big\}.
\end{equation*}
Observe that $\big|\Omega_{\lambda,j}\big|\le \big|\Omega^*_\lambda\big|+\big|\Omega_{\lambda,j}\cap(\mathbb R^n\backslash\Omega^*_\lambda)\big|$. Thus
\begin{equation*}
\begin{split}
\big\|\mathcal S_{\alpha,2^j}(f)\big\|^p_{L^p}&=\int_0^\infty p\lambda^{p-1}\big|\Omega_{\lambda,j}\big|\,d\lambda\\
&\le\int_0^\infty p\lambda^{p-1}\big|\Omega^*_\lambda\big|\,d\lambda+\int_0^\infty p\lambda^{p-1}\big|\Omega_{\lambda,j}\cap(\mathbb R^n\backslash\Omega^*_\lambda)\big|\,d\lambda\\
&=\mbox{\upshape I+II}.
\end{split}
\end{equation*}
The weak type (1,1) estimate of $M$ yields
\begin{equation}
\mbox{\upshape I}\le C\cdot2^{jn}\int_0^\infty p\lambda^{p-1}|\Omega_\lambda|\,d\lambda=C\cdot2^{jn}\big\|\mathcal S_\alpha(f)\big\|^p_{L^p}.
\end{equation}
To estimate II, we now claim that the following inequality holds.
\begin{equation}
\int_{\mathbb R^n\backslash\Omega^*_\lambda}\mathcal S_{\alpha,2^j}(f)(x)^2\,dx\le C\cdot2^{jn}\int_{\mathbb R^n\backslash\Omega_\lambda}\mathcal S_{\alpha}(f)(x)^2\,dx.
\end{equation}
We will take the above inequality temporarily for granted, then it follows from Chebyshev's inequality and (4.4) that
\begin{equation*}
\begin{split}
\big|\Omega_{\lambda,j}\cap(\mathbb R^n\backslash\Omega^*_\lambda)\big|&\le\lambda^{-2}\int_{\Omega_{\lambda,j}\cap(\mathbb R^n\backslash\Omega^*_\lambda)}\mathcal S_{\alpha,2^j}(f)(x)^2\,dx\\
&\le\lambda^{-2}\int_{\mathbb R^n\backslash\Omega^*_\lambda}\mathcal S_{\alpha,2^j}(f)(x)^2\,dx\\
&\le C\cdot2^{jn}\lambda^{-2}\int_{\mathbb R^n\backslash\Omega_\lambda}\mathcal S_{\alpha}(f)(x)^2\,dx.
\end{split}
\end{equation*}
Hence
\begin{equation*}
\mbox{\upshape II}\le C\cdot2^{jn}\int_0^\infty p\lambda^{p-1}\bigg(\lambda^{-2}\int_{\mathbb R^n\backslash\Omega_\lambda}\mathcal S_{\alpha}(f)(x)^2\,dx\bigg)d\lambda.
\end{equation*}
Changing the order of integration yields
\begin{align}
\mbox{\upshape II}&\le C\cdot2^{jn}\int_{\mathbb R^n}\mathcal S_\alpha(f)(x)^2\bigg(\int_{|S_\alpha(f)(x)|}^\infty p\lambda^{p-3}\,d\lambda\bigg)\,dx\notag\\
&\le C\cdot2^{jn}\frac{p}{2-p}\cdot\big\|\mathcal S_\alpha(f)\big\|^p_{L^p}.
\end{align}
Combining the above estimate (4.5) with (4.3) and taking $p$-th root on both sides, we complete the proof of Lemma 4.3. So it remains to prove the inequality (4.4). Set $\Gamma_{2^j}(\mathbb R^n\backslash\Omega^*_\lambda)=\underset{x\in\mathbb R^n\backslash\Omega^*_\lambda}{\bigcup}\Gamma_{2^j}(x)$ and
$\Gamma(\mathbb R^n\backslash\Omega_\lambda)=\underset{x\in\mathbb R^n\backslash\Omega_\lambda}{\bigcup}\Gamma(x).$
For each given $(y,t)\in\Gamma_{2^j}(\mathbb R^n\backslash\Omega^*_\lambda)$, we have
\begin{equation*}
\big|B(y,2^jt)\cap(\mathbb R^n\backslash\Omega^*_\lambda)\big|\le 2^{jn}\big|B(y,t)\big|.
\end{equation*}
It is not difficult to check that $\big|B(y,t)\cap\Omega_\lambda\big|\le\frac{|B(y,t)|}{2}$ and $\Gamma_{2^j}(\mathbb R^n\backslash\Omega^*_\lambda)\subseteq\Gamma(\mathbb R^n\backslash\Omega_\lambda)$. In fact, for any $(y,t)\in\Gamma_{2^j}(\mathbb R^n\backslash\Omega^*_\lambda)$, there exists a point $x\in \mathbb R^n\backslash\Omega^*_\lambda$ such that $(y,t)\in\Gamma_{2^j}(x)$. Then we can deduce
\begin{equation*}
\begin{split}
\big|B(y,t)\cap\Omega_\lambda\big|&\le \big|B(y,2^jt)\cap\Omega_\lambda\big|\\
&= \int_{B(y,2^jt)}\chi_{\Omega_\lambda}(z)\,dz\\
&\le 2^{jn}|B(y,t)|\cdot\frac{1}{|B(y,2^jt)|}\int_{B(y,2^jt)}\chi_{\Omega_\lambda}(z)\,dz.
\end{split}
\end{equation*}
Note that $x\in B(y,2^jt)\cap(\mathbb R^n\backslash\Omega^*_\lambda)$. So we have
\begin{equation*}
\begin{split}
\big|B(y,t)\cap\Omega_\lambda\big|\le 2^{jn}|B(y,t)|\cdot M(\chi_{\Omega_\lambda})(x)\le \frac{|B(y,t)|}{2}.
\end{split}
\end{equation*}
Consequently
\begin{equation*}
\begin{split}
\big|B(y,t)\big|&=\big|B(y,t)\cap\Omega_\lambda\big|+\big|B(y,t)\cap(\mathbb R^n\backslash\Omega_\lambda)\big|\\
&\le \frac{|B(y,t)|}{2}+\big|B(y,t)\cap(\mathbb R^n\backslash\Omega_\lambda)\big|,
\end{split}
\end{equation*}
which is equivalent to
\begin{equation*}
\big|B(y,t)\big|\le 2\cdot \big|B(y,t)\cap(\mathbb R^n\backslash\Omega_\lambda)\big|.
\end{equation*}
The above inequality implies in particular that there is a point $z\in B(y,t)\cap(\mathbb R^n\backslash\Omega_\lambda)\neq\emptyset$. In this case, we have $(y,t)\in\Gamma(z)$ with $z\in \mathbb R^n\backslash\Omega_\lambda$, which gives $\Gamma_{2^j}(\mathbb R^n\backslash\Omega^*_\lambda)\subseteq\Gamma(\mathbb R^n\backslash\Omega_\lambda)$. Thus we obtain
\begin{equation*}
\big|B(y,2^jt)\cap(\mathbb R^n\backslash\Omega_\lambda^*)\big|\le C\cdot2^{jn}\big|B(y,t)\cap(\mathbb R^n\backslash\Omega_\lambda)\big|.
\end{equation*}
Therefore
\begin{equation*}
\begin{split}
&\int_{\mathbb R^n\backslash\Omega^*_\lambda}\mathcal S_{\alpha,2^j}(f)(x)^2\,dx\\
=&\int_{\mathbb R^n\backslash\Omega^*_\lambda}\bigg(\iint_{\Gamma_{2^j}(x)}\Big(A_\alpha(f)(y,t)\Big)^2\frac{dydt}{t^{n+1}}
\bigg)\,dx\\
\le&\iint_{\Gamma_{2^j}(\mathbb R^n\backslash\Omega^*_\lambda)}\bigg(\int_{B(y,2^jt)\cap(\mathbb R^n\backslash\Omega_\lambda^*)}\,dx\bigg)\Big(A_\alpha(f)(y,t)\Big)^2\frac{dydt}{t^{n+1}}\\
\le&\,C\cdot2^{jn}\iint_{\Gamma(\mathbb R^n\backslash\Omega_\lambda)}\bigg(\int_{B(y,t)\cap(\mathbb R^n\backslash\Omega_{\lambda})}\,dx\bigg)\Big(A_\alpha(f)(y,t)\Big)^2\frac{dydt}{t^{n+1}}\\
\le&\,C\cdot2^{jn}\int_{\mathbb R^n\backslash\Omega_\lambda}\mathcal S_{\alpha}(f)(x)^2\,dx.
\end{split}
\end{equation*}
This finishes the proof of the Lemma 4.3.
\end{proof}

We are now in a position to give the proof of Theorem 1.4.

\begin{proof}[Proof of Theorem 1.4]
From the definition of $g^*_{\lambda,\alpha}$, we readily see that
\begin{align}
g^*_{\lambda,\alpha}(f)(x)^2=&\iint_{\mathbb R^{n+1}_+}\left(\frac{t}{t+|x-y|}\right)^{\lambda n}\Big(A_\alpha(f)(y,t)\Big)^2\frac{dydt}{t^{n+1}}\notag\\
=&\int_0^\infty\int_{|x-y|<t}\left(\frac{t}{t+|x-y|}\right)^{\lambda n}\Big(A_\alpha(f)(y,t)\Big)^2\frac{dydt}{t^{n+1}}\notag\\
&+\sum_{j=1}^\infty\int_0^\infty\int_{2^{j-1}t\le|x-y|<2^jt}\left(\frac{t}{t+|x-y|}\right)^{\lambda n}\Big(A_\alpha(f)(y,t)\Big)^2\frac{dydt}{t^{n+1}}\notag\\
\le&\, C\bigg[\mathcal S_\alpha(f)(x)^2+\sum_{j=1}^\infty 2^{-j\lambda n}\mathcal S_{\alpha,2^j}(f)(x)^2\bigg].
\end{align}
Let $f\in L^{p,\Phi}$ with $1<p<\infty$. For any given ball $B=B(x_0,r)\subseteq\mathbb R^n$, then from the above inequality (4.6), it follows that
\begin{equation*}
\begin{split}
&\frac{1}{\Phi(r)^{1/p}}\bigg(\int_B\big|g^*_{\lambda,\alpha}(f)(x)\big|^p\,dx\bigg)^{1/p}\\
\end{split}
\end{equation*}
\begin{equation*}
\begin{split}
\le\,&\frac{1}{\Phi(r)^{1/p}}\bigg(\int_B\big|\mathcal S_\alpha(f)(x)\big|^p\,dx\bigg)^{1/p}
+\sum_{j=1}^\infty 2^{-j\lambda n/2}\cdot
\frac{1}{\Phi(r)^{1/p}}\bigg(\int_B\big|\mathcal S_{\alpha,2^j}(f)(x)\big|^p\,dx\bigg)^{1/p}\\
=\,&I_0+\sum_{j=1}^\infty 2^{-j\lambda n/2}I_j.
\end{split}
\end{equation*}
By Theorem 1.1, we know that $I_0\le C\|f\|_{L^{p,\Phi}}$. Below we shall give the estimates of $I_j$ for $j=1,2,\ldots.$ As before, we set $f=f_1+f_2$, $f_1=f\chi_{_{2B}}$ and write
\begin{equation*}
\begin{split}
I_j&\le \frac{1}{\Phi(r)^{1/p}}\bigg(\int_{B(x_0,r)}\big|\mathcal S_{\alpha,2^j}(f_1)(x)\big|^p\,dx\bigg)^{1/p}+
\frac{1}{\Phi(r)^{1/p}}\bigg(\int_{B(x_0,r)}\big|\mathcal S_{\alpha,2^j}(f_2)(x)\big|^p\,dx\bigg)^{1/p}\\
&=I^{(1)}_j+I^{(2)}_j.
\end{split}
\end{equation*}
Applying Lemmas 4.1--4.3, Theorem A and the condition (1.1), we obtain
\begin{equation*}
\begin{split}
I^{(1)}_j&\le \frac{1}{\Phi(r)^{1/p}}\big\|\mathcal S_{\alpha,2^j}(f_1)\big\|_{L^p}\\
&\le C\Big(2^{{jn}/2}+2^{{jn}/p}\Big)\frac{1}{\Phi(r)^{1/p}}\cdot\big\|\mathcal S_\alpha(f_1)\big\|_{L^p}\\
&\le C\Big(2^{{jn}/2}+2^{{jn}/p}\Big)\frac{1}{\Phi(r)^{1/p}}\cdot\|f_1\|_{L^p}\\
&\le C\|f\|_{L^{p,\Phi}}\Big(2^{{jn}/2}+2^{{jn}/p}\Big)\cdot\frac{\Phi(2r)^{1/p}}{\Phi(r)^{1/p}}\\
&\le C\|f\|_{L^{p,\Phi}}\Big(2^{{jn}/2}+2^{{jn}/p}\Big).
\end{split}
\end{equation*}
We now turn to estimate the other term $I^{(2)}_j$. For any $x\in B$, $(y,t)\in\Gamma_{2^j}(x)$ and $z\in\big(2^{k+1}B\backslash 2^{k}B\big)\cap B(y,t)$, then by a direct calculation, we can easily deduce
\begin{equation*}
t+2^j t\ge |x-y|+|y-z|\ge|x-z|\ge|z-x_0|-|x-x_0|\ge 2^{k-1}r.
\end{equation*}
Thus, it follows from the previous estimates (2.1), (2.3) and Minkowski's integral inequality that
\begin{equation*}
\begin{split}
\big|\mathcal S_{\alpha,2^j}(f_2)(x)\big|&=\left(\iint_{\Gamma_{2^j}(x)}\sup_{\varphi\in{\mathcal C}_\alpha}|f_2*\varphi_t(y)|^2\frac{dydt}{t^{n+1}}\right)^{1/2}\\
&\le C\left(\int_{2^{(k-2-j)}r}^\infty\int_{|x-y|<2^jt}\bigg|t^{-n}\sum_{k=1}^\infty\int_{2^{k+1}B\backslash 2^{k}B}|f(z)|\,dz\bigg|^2\frac{dydt}{t^{n+1}}\right)^{1/2}\\
&\le C\bigg(\sum_{k=1}^\infty\int_{2^{k+1}B\backslash 2^{k}B}|f(z)|\,dz\bigg)\bigg(\int_{2^{(k-2-j)}r}^\infty 2^{jn}\frac{dt}{t^{2n+1}}\bigg)^{1/2}\\
&\le C\cdot2^{{3jn}/2}\sum_{k=1}^\infty\frac{1}{|B(x_0,2^{k+1}r)|}\int_{2^{k+1}B\backslash 2^{k}B}|f(z)|\,dz\\
&\le C\|f\|_{L^{p,\Phi}}\cdot 2^{{3jn}/2}\sum_{k=1}^\infty\frac{\Phi(2^{k+1}r)^{1/p}}{|B(x_0,2^{k+1}r)|^{1/p}}.
\end{split}
\end{equation*}
Furthermore, by using the inequality (2.5) again, we have
\begin{equation*}
\begin{split}
I^{(2)}_j&\le C\|f\|_{L^{p,\Phi}}\cdot2^{{3jn}/2}\sum_{k=1}^\infty\frac{|B(x_0,r)|^{1/p}}{\Phi(r)^{1/p}}
\cdot\frac{\Phi(2^{k+1}r)^{1/p}}{|B(x_0,2^{k+1}r)|^{1/p}}\\
&\le C\|f\|_{L^{p,\Phi}}\cdot2^{{3jn}/2}.
\end{split}
\end{equation*}
Therefore
\begin{equation*}
\begin{split}
&\frac{1}{\Phi(r)^{1/p}}\bigg(\int_{B(x_0,r)}\big|g^*_{\lambda,\alpha}(f)(x)\big|^p\,dx\bigg)^{1/p}\\
\le&\, C\|f\|_{L^{p,\Phi}}
\left(1+\sum_{j=1}^\infty 2^{-j\lambda n/2}2^{{3jn}/2}+\sum_{j=1}^\infty 2^{-j\lambda n/2}2^{jn/p}\right)\\
\le&\,  C\|f\|_{L^{p,\Phi}},
\end{split}
\end{equation*}
where the last two series are both convergent under our assumption $\lambda>3>2/p$ and $p>1$. Hence, by taking the supremum over all balls $B=B(x_0,r)\subseteq\mathbb R^n$, we conclude the proof of Theorem 1.4.
\end{proof}

\section{Proof of Theorem 1.5}

Let us first prove the following result.

\begin{theorem}
Let $0<\alpha\le1$ and $\lambda>{(3n+2\alpha)}/n$. Then for any $\sigma>0$, there exists a constant $C>0$ independent of $f$ and $\sigma$ such that
\begin{equation*}
\big|\big\{x\in\mathbb R^n: g^*_{\lambda,\alpha}(f)(x)>\sigma\big\}\big|\le \frac{C}{\sigma}\int_{\mathbb R^n}|f(x)|\,dx.
\end{equation*}
\end{theorem}

\begin{proof}
For any given $\sigma>0$ and $f\in L^1(\mathbb R^n)$, we apply the Calder\'on-Zygmund decomposition of $f$ at level $\sigma$ to obtain a sequence of disjoint non-overlapping dyadic cubes $\{Q_i\}$ and two functions $g$, $b$ such that the following properties hold: (see \cite{stein})

$(i)$ $f(x)=g(x)+b(x)$;

$(ii)$ $\|g\|^2_{L^2}\le C\cdot\sigma\|f\|_{L^1}$;

$(iii)$ $b(x)=0$, \quad a.e. $x\in\mathbb R^n\backslash\bigcup_i Q_i$;

$(iv)$  $\int_{Q_i}b(x)\,dx=0$;

$(v)$  $\sum_i|Q_i|\le \sigma^{-1}\|f\|_{L^1}$.

By the previous inequality (4.6), we write
\begin{equation*}
\begin{split}
&\big|\big\{x\in \mathbb R^n:|g^*_{\lambda,\alpha}(f)(x)|>\sigma\big\}\big|\\
\le&\,\big|\big\{x\in \mathbb R^n:|\mathcal S_{\alpha}(f)(x)|>\sigma/2\big\}\big|+\Big|\Big\{x\in \mathbb R^n:\Big|\sum_{j=1}^\infty 2^{-j\lambda n/2}\mathcal S_{\alpha,2^j}(f)(x)\Big|>\sigma/2\Big\}\Big|\\
=&\,\mbox{\upshape I+II}.
\end{split}
\end{equation*}
Using Theorem B, we have
\begin{equation*}
\mbox{I}\le\frac{C}{\sigma}\int_{\mathbb R^n}|f(x)|\,dx.
\end{equation*}
Now for $j=1,2,\ldots$, since $\mathcal S_{\alpha,2^j}(f)(x)\le \mathcal S_{\alpha,2^j}(g)(x)+\mathcal S_{\alpha,2^j}(b)(x)$ by the property $(i)$, then it follows that
\begin{equation*}
\begin{split}
\mbox{II}\le&\, \Big|\Big\{x\in \mathbb R^n:\Big|\sum_{j=1}^\infty 2^{-j\lambda n/2}\mathcal S_{\alpha,2^j}(g)(x)\Big|>\sigma/4\Big\}\Big|\\
&+\Big|\Big\{x\in \mathbb R^n:\Big|\sum_{j=1}^\infty 2^{-j\lambda n/2}\mathcal S_{\alpha,2^j}(b)(x)\Big|>\sigma/4\Big\}\Big|\\
=&\,\mbox{\upshape III+IV}.
\end{split}
\end{equation*}
Applying Minkowski's inequality, Lemma 4.1, Theorem A and the property $(ii)$, we can deduce
\begin{equation*}
\begin{split}
\mbox{\upshape III}&\le\frac{C}{\sigma^2}\bigg\|\sum_{j=1}^\infty 2^{-j\lambda n/2}\mathcal S_{\alpha,2^j}(g)\bigg\|^2_{L^2}\\
&\le\frac{C}{\sigma^2}\bigg(\sum_{j=1}^\infty 2^{-j\lambda n/2}\cdot 2^{jn/2}\|g\|_{L^2}\bigg)^2\\
&\le\frac{C}{\sigma^2}\cdot\|g\|^2_{L^2}\\
&\le\frac{C}{\sigma}\cdot\|f\|_{L^1}.
\end{split}
\end{equation*}
To estimate IV, let $Q_i^*=2\sqrt n Q_i$ be a cube whose center is the same as $Q_i$ and side is $2\sqrt n$ times that of $Q_i$. Then we can further decompose IV as follows.
\begin{equation*}
\begin{split}
\mbox{\upshape IV}\le&\,\Big|\Big\{x\in \bigcup_i Q_i^*:\Big|\sum_{j=1}^\infty 2^{-j\lambda n/2}\mathcal S_{\alpha,2^j}(b)(x)\Big|>\sigma/4\Big\}\Big|\\
&+\Big|\Big\{x\notin \bigcup_i Q_i^*:\Big|\sum_{j=1}^\infty 2^{-j\lambda n/2}\mathcal S_{\alpha,2^j}(b)(x)\Big|>\sigma/4\Big\}\Big|\\
=&\,\mbox{\upshape IV}^{(1)}+\mbox{\upshape IV}^{(2)}.
\end{split}
\end{equation*}
It follows immediately from the property $(v)$ that
\begin{equation*}
\mbox{\upshape IV}^{(1)}\le\sum_i\big|Q_i^*\big|\le C\sum_i|Q_i|\le\frac{C}{\sigma}\cdot\|f\|_{L^1}.
\end{equation*}
We set
\begin{equation*}
b_i(x)=
\begin{cases}
b(x) &  \mbox{if}\;\; x\in Q_i,\\
0    &  \mbox{if}\;\; x\notin Q_i.
\end{cases}
\end{equation*}
Then by the properties $(iii)$ and $(iv)$, we have $b(x)=\sum_i b_i(x)$, $supp\,b_i\subseteq Q_i$, $\int_{Q_i}b_i(x)\,dx=0$ and $\|b_i\|_{L^1}\le 2\int_{Q_i}|f(x)|\,dx$.
For any $\varphi\in{\mathcal C}_\alpha$, $0<\alpha\le1$, by the vanishing moment condition of $b_i$, we have that for any $(y,t)\in\Gamma_{2^j}(x)$,
\begin{align}
\big|(b_i*\varphi_t)(y)\big|&=\left|\int_{Q_i}\big(\varphi_t(y-z)-\varphi_t(y-x_0)\big)b_i(z)\,dz\right|\notag\\
&\le\int_{Q_i\cap\{z:|z-y|\le t\}}\frac{|z-x_0|^\alpha}{t^{n+\alpha}}|b_i(z)|\,dz\notag\\
&\le C\cdot\frac{|Q_i|^{\alpha/n}}{t^{n+\alpha}}\int_{Q_i\cap\{z:|z-y|\le t\}}|b_i(z)|\,dz.
\end{align}
Denote the center of $Q_i$ by $c_i$. Then for any $z\in Q_i$ and $x\in (Q^*_i)^c$, we have $|z-c_i|<\frac{|x-c_i|}{2}$. Thus, for all $(y,t)\in\Gamma_{2^j}(x)$ and $|z-y|\le t$ with $z\in Q_i$, we can deduce that
\begin{equation}
t+2^jt\ge|x-y|+|y-z|\ge|x-z|\ge|x-c_i|-|z-c_i|\ge\frac{|x-c_i|}{2}.
\end{equation}
Therefor, for any $x\in (Q^*_i)^c$, by using the above inequalities (5.1) and (5.2), we obtain
\begin{equation*}
\begin{split}
\big|\mathcal S_{\alpha,2^j}(b_i)(x)\big|&=\left(\iint_{\Gamma_{2^j}(x)}\Big(\sup_{\varphi\in{\mathcal C}_\alpha}\big|(\varphi_t*{b_i})(y)\big|\Big)^2\frac{dydt}{t^{n+1}}\right)^{1/2}\\
&\le C\cdot\big|Q_i\big|^{\alpha/n}\bigg(\int_{Q_i}|b_i(z)|\,dz\bigg)\left(\int_{\frac{|x-c_i|}{2^{j+2}}}^\infty
\int_{|y-x|<2^jt}\frac{dydt}{t^{2(n+\alpha)+n+1}}\right)^{1/2}\\
&\le C\cdot2^{{jn}/2}\big|Q_i\big|^{\alpha/n}\bigg(\int_{Q_i}|b_i(z)|\,dz\bigg)
\left(\int_{\frac{|x-c_i|}{2^{j+2}}}^\infty\frac{dt}{t^{2(n+\alpha)+1}}\right)^{1/2}\\
&\le C\cdot2^{j(3n+2\alpha)/2}\frac{|Q_i|^{\alpha/n}}{|x-c_i|^{n+\alpha}}\bigg(\int_{Q_i}|b_i(z)|\,dz\bigg).
\end{split}
\end{equation*}
Hence, by our hypothesis $\lambda>{(3n+2\alpha)}/n$, we have
\begin{equation*}
\begin{split}
\mbox{\upshape IV}^{(2)}&\le\frac{4}{\sigma}\int_{\mathbb R^n\backslash\bigcup_i Q_i^*}\Big|\sum_{j=1}^\infty 2^{-j\lambda n/2}\mathcal S_{\alpha,2^j}(b)(x)\Big|dx\\
&\le\frac{4}{\sigma}\sum_{j=1}^\infty 2^{-j\lambda n/2}\sum_i\left(\int_{(Q_i^*)^c}\mathcal S_{\alpha,2^j}(b_i)(x)dx\right)\\
&\le\frac{C}{\sigma}\left(\sum_{j=1}^\infty 2^{-j\lambda n/2}\cdot2^{j(3n+2\alpha)/2}\right)
\left(\sum_i|Q_i|^{\alpha/n}\|b_i\|_{L^1}\int_{(Q_i^*)^c}\frac{dx}{|x-c_i|^{n+\alpha}}\right)\\
&\le\frac{C}{\sigma}\cdot\sum_i\int_{Q_i}|f(z)|\,dz\\
&\le\frac{C}{\sigma}\cdot\|f\|_{L^1}.
\end{split}
\end{equation*}
Summing up the above estimates, we finish the proof of Theorem 5.1.
\end{proof}

We are now ready to prove Theorem 1.5.

\begin{proof}[Proof of Theorem 1.5]
Let $f\in L^{1,\Phi}$. For each fixed ball $B=B(x_0,r)\subseteq\mathbb R^n$, we again decompose $f$ as $f=f_1+f_2$, where $f_1=f\chi_{_{2B}}$. For any given $\sigma>0$, then we write
\begin{equation*}
\begin{split}
&\big|\big\{x\in B(x_0,r):|g^*_{\lambda,\alpha}(f)(x)|>\sigma\big\}\big|\\
\le\,& \big|\big\{x\in B(x_0,r):|g^*_{\lambda,\alpha}(f_1)(x)|>\sigma/2\big\}\big|+\big|\big\{x\in B(x_0,r):|g^*_{\lambda,\alpha}(f_2)(x)|>\sigma/2\big\}\big|\\
  =\,&J'_1+J'_2.
\end{split}
\end{equation*}
Theorem 5.1 and the condition (1.1) imply
\begin{equation*}
\begin{split}
J'_1&\le\frac{C}{\sigma}\int_{2B}|f(y)|\,dy\\
&\le\frac{C\cdot\Phi(2r)}{\sigma}\|f\|_{L^{1,\Phi}}\\
&\le\frac{C\cdot\Phi(r)}{\sigma}\|f\|_{L^{1,\Phi}}.
\end{split}
\end{equation*}
For the term $J'_2$, note that in the proofs of Theorems 1.2 and 1.4, we have already showed that for any $x\in B(x_0,r)$,
\begin{equation}
\big|\mathcal S_{\alpha}(f_2)(x)\big|\le C\|f\|_{L^{1,\Phi}}\cdot\frac{\Phi(r)}{|B(x_0,r)|}.
\end{equation}
and
\begin{equation*}
\begin{split}
\big|\mathcal S_{\alpha,2^j}(f_2)(x)\big|\le C\cdot2^{{3jn}/2}\sum_{k=1}^\infty\frac{1}{|B(x_0,2^{k+1}r)|}\int_{2^{k+1}B\backslash 2^{k}B}|f(z)|\,dz.
\end{split}
\end{equation*}
Moreover, it follows directly from the inequality (2.6) that
\begin{align}
\big|\mathcal S_{\alpha,2^j}(f_2)(x)\big|&\le C\|f\|_{L^{1,\Phi}}\cdot2^{{3jn}/2}\frac{\Phi(r)}{|B(x_0,r)|}
\sum_{k=1}^\infty\frac{|B(x_0,r)|}{\Phi(r)}\cdot\frac{\Phi(2^{k+1}r)}{|B(x_0,2^{k+1}r)|}\notag\\
&\le C\|f\|_{L^{1,\Phi}}\cdot2^{{3jn}/2}\frac{\Phi(r)}{|B(x_0,r)|}.
\end{align}
Therefore, by using the estimates (4.6), (5.3) and (5.4), we get
\begin{equation*}
\begin{split}
\big|g^*_{\lambda,\alpha}(f_2)(x)\big|&\le C\left(\big|\mathcal S_\alpha(f_2)(x)\big|+\sum_{j=1}^\infty 2^{{-j\lambda n}/2}\big|\mathcal S_{\alpha,2^j}(f_2)(x)\big|\right)\\
&\le C\|f\|_{L^{1,\Phi}}\cdot\frac{\Phi(r)}{|B(x_0,r)|}\left(1+\sum_{j=1}^\infty2^{{-j\lambda n}/2}\cdot2^{{3jn}/2}\right)\\
&\le C\|f\|_{L^{1,\Phi}}\cdot\frac{\Phi(r)}{|B(x_0,r)|},
\end{split}
\end{equation*}
where the last series is convergent since $\lambda>{(3n+2\alpha)}/n>3$. The rest of the proof is exactly the same as that of Theorem 1.2, and we finally obtain
\begin{equation*}
J'_2\le \frac{C\cdot\Phi(r)}{\sigma}\|f\|_{L^{1,\Phi}}.
\end{equation*}
Combining the above estimates for $J'_1$ and $J'_2$ and taking the supremum over all balls $B(x_0,r)\subseteq\mathbb R^n$ and all $\sigma>0$, we conclude the proof of Theorem 1.5.
\end{proof}

Finally, we remark that for a given real-valued function $b\in BMO(\mathbb R^n)$, by the preceding estimate (3.1), we can see that
\begin{equation*}
\big|\big[b,g^*_{\lambda,\alpha}\big](f)(x)\big|\le\frac{1}{2\pi}\int_0^{2\pi}
g^*_{\lambda,\alpha}\big(e^{-e^{i\theta}b}\cdot f\big)(x)\cdot e^{\cos\theta\cdot b(x)}\,d\theta.
\end{equation*}
Following the idea in \cite{alvarez} and \cite{ding}, we can also prove

\begin{theorem}
Let $0<\alpha\le1$, $1<p<\infty$ and $\lambda>3$. Then the commutator $\big[b,g^*_{\lambda,\alpha}\big]$ is bounded from $L^p(\mathbb R^n)$ into itself whenever $b\in BMO(\mathbb R^n)$.
\end{theorem}

Thus, by using the same arguments as in the proof of Theorems 1.3 and 1.4, we can also show the conclusion of Theorem 1.6. The details are omitted here.


\begin{thebibliography}{99}

\bibitem{adams} D. R. Adams, A note on Riesz potentials, Duke Math. J, \textbf{42}(1975), 765--778.
\bibitem{alvarez} J. Alvarez, R. J. Bagby, D. S. Kurtz and C. P\'erez, Weighted estimates for commutators of linear operators, Studia Math, \textbf{104}(1993), 195--209.
\bibitem{chang} S. Y. A. Chang, J. M. Wilson and T. H. Wolff, Some weighted norm inequalities concerning the Schr\"odinger operators, Comment. Math. Helv,  \textbf{60}(1985), 217--246.
\bibitem{chanillo} S. Chanillo and R. L. Wheeden, Some weighted norm inequalities for the area integral, Indiana Univ. Math. J, \textbf{36}(1987), 277--294.
\bibitem{chiarenza} F. Chiarenza and M. Frasca, Morrey spaces and Hardy-Littlewood maximal function, Rend. Math. Appl, \textbf{7}(1987), 273--279.
\bibitem{ding} Y. Ding, S. Z. Lu and K. Yabuta, On commutators of Marcinkiewicz integrals with rough kernel, J. Math. Anal. Appl, \textbf{275}(2002), 60--68.
\bibitem{duoand} J. Duoandikoetxea, Fourier Analysis, American Mathematical Society, Providence, Rhode Island, 2000.
\bibitem{fan} D. S. Fan, S. Z. Lu and D. C. Yang, Regularity in Morrey spaces of strong solutions to nondivergence
    elliptic equations with VMO coefficients, Georgian Math. J, \textbf{5}(1998), 425--440.
\bibitem{fazio1} G. Di Fazio and M. A. Ragusa, Interior estimates in Morrey spaces for strong solutions to
    nondivergence form equations with discontinuous coefficients, J. Funct. Anal, \textbf{112}(1993), 241--256.
\bibitem{fazio2} G. Di Fazio, D. K. Palagachev and M. A. Ragusa, Global Morrey regularity of strong solutions to the
    Dirichlet problem for elliptic equations with discontinuous coefficients, J. Funct. Anal, \textbf{166}(1999),
    179--196.
\bibitem{huang} J. Z. Huang and Y. Liu, Some characterizations of weighted Hardy spaces, J. Math. Anal. Appl, \textbf{363}(2010), 121--127.
\bibitem{john} F. John and L. Nirenberg, On functions of bounded mean oscillation, Comm. Pure Appl. Math,
    \textbf{14}(1961), 415--426.
\bibitem{lu} S. Z. Lu, D. C. Yang and Z. S. Zhou, Sublinear operators with rough kernel on generalized Morrey spaces, Hokkaido Math. J, \textbf{27}(1998), 219--232.
\bibitem{mizuhara} T. Mizuhara, Boundedness of some classical operators on generalized Morrey spaces, Harmonic Analysis, ICM-90 Satellite Conference Proceedings, Springer-Verlag, Tokyo, (1991), 183--189.
\bibitem{morrey} C. B. Morrey, On the solutions of quasi-linear elliptic partial differential equations, Trans. Amer.
    Math. Soc, \textbf{43}(1938), 126--166.
\bibitem{peetre} J. Peetre, On the theory of $\mathcal L_{p,\lambda}$ spaces, J. Funct. Anal, \textbf{4}(1969),
    71--87.
\bibitem{stein} E. M. Stein, Singular Integrals and Differentiability Properties of Functions, Princeton Univ. Press, Princeton, New Jersey, 1970.
\bibitem{torchinsky} A. Torchinsky, Real-Variable Methods in Harmonic Analysis, Academic Press, New York, 1986.
\bibitem{wang4} H. Wang, Boundedness of intrinsic square functions on the weighted weak Hardy spaces, preprint, 2012.
\bibitem{wang1} H. Wang, Intrinsic square functions on the weighted Morrey spaces, J. Math. Anal. Appl, to appear.
\bibitem{wang2} H. Wang and H. P. Liu, The intrinsic square function characterizations of weighted Hardy spaces, Illinois J. Math, to appear.
\bibitem{wang3} H. Wang and H. P. Liu, Weak type estimates of intrinsic square functions on the weighted Hardy spaces, Arch. Math., \textbf{97}(2011), 49--59.
\bibitem{wilson1} M. Wilson, The intrinsic square function, Rev. Mat. Iberoamericana, \textbf{23}(2007), 771--791.
\bibitem{wilson2} M. Wilson, Weighted Littlewood-Paley Theory and Exponential-Square Integrability, Lecture Notes in Math, Vol 1924, Springer-Verlag, 2007.

\end{thebibliography}
\end{document}